\documentclass[11pt,reqno]{amsart}
\usepackage{amsthm,amssymb,amsmath}
\usepackage{textcomp}
\usepackage[foot]{amsaddr}
\usepackage{xcolor}
\usepackage{hyperref}
\usepackage[T1]{fontenc}
\usepackage{lmodern}
\hypersetup{
	colorlinks =true , 
	linkcolor =blue,
	urlcolor = magenta,
	citecolor =red}

\newtheorem{theorem}{Theorem}[section]
\newtheorem*{theorem*}{Theorem}
\newtheorem{lemma}{Lemma}[section]
\newtheorem{corollary}{Corollary}[section]

\newtheorem{claim}{\sc Claim}

\theoremstyle{definition}

\newtheorem{definition}{\bf Definition}[section]
\newtheorem*{definition*}{\bf Definition}

\newtheorem{example}{\bf Example}[section]

\newtheorem{remark}{Remark}
\newtheorem*{remark*}{Remark}

\newtheorem*{example*}{\bf Example}

\begin{document}

	\title[Semilinear elliptic systems]{An infinite-dimensional mountain pass theorem with applications to nonlinear elliptic systems}

	\author{Ablanvi Songo$^{1,a}$}

\address{$^a$Universit\'{e} de Sherbrooke, D\'{e}partement de math\'{e}matiques, Sherbrooke, Qu\'{e}bec, Canada}

\email{\href{mailto:ablanvi.songo@usherbrooke.ca}{{\textcolor{blue}{ablanvi.songo@usherbrooke.ca}}}}
\footnote{Corresponding author}
\author{ Fabrice Colin${^b}$}

\address{$^b$Laurentian University, School of Engineering and Computer Science, Sudbury, Ontario, Canada}

\email{\href{mailto:fcolin@laurentian.ca}{{\textcolor{blue}{fcolin@laurentian.ca}}}}

	\noindent
	\keywords{Semilinear Hamiltonian system, strongly indefinite functionals,  generalized Mountain Pass Theorem}
	
	\subjclass[2020]{35A15; 35J61; 58E05}

	\begin{abstract}
		The purpose of this paper is to establish a critical point theorem, which is an infinite-dimensional generalization of the classical generalized Mountain Pass Theorem of P. H. Rabinowitz \cite[Theorem 5.3]{Ra}. As an application, we obtain the existence of at least one nontrivial solution to a semilinear elliptic systems with indefinite weights in $\mathbb{R}^2$.
	\end{abstract}

	\maketitle

	\tableofcontents

	\section{Introduction}
    In \cite{Ra 2} (see also \cite[Theorem 5.3]{Ra}), P. H. Rabinowitz stated the following generalization of the classical Mountain Pass Theorem:
    \begin{theorem}
    \label{theo 1.1}
        Let $\Big(E, \|\cdot\|\Big)$ be a real Banach space with $E=V\oplus W$, where $V$ is finite dimensional. Suppose $I\in\mathcal{C}^1(E, \mathbb{R})$, satisfies $(PS)$, that is, any sequence $(u_m) \subset E$ for which $I(u_m)$ is bounded and $I'(u_m)\to 0$ as $m\to \infty$ possesses a convergent subsequence. Suppose also that $I$ satisfies 
        \begin{enumerate}
            \item[$(I_1)$] there are constants $\rho$, $\alpha>0$ such that $I(u)\ge \alpha$ on $\Big\{ u\in W\;|\; \|u\|=\rho\Big\}$, 
            \item[$(I_2)$] there is an $e\in W$ such that $\|e\|=1$ and $R>\rho$ such that if \[M:= \Big\{ u\in V \;|\; \|u\|\le R\Big\}\oplus \Big\{ re\;|\; 0\le r\le R \Big\},\] then $\underset{u\in \partial M}{\max}\; I(u)\le0$. 
        \end{enumerate}
        Then $I$ possesses a critical value $c\ge  \alpha$ which can be characterized as
        \[c:= \underset{h\in \Gamma}{\inf}\; \underset{u\in  M}{\max}\; I(h(u)),\]
        where \[\Gamma := \Big\{ h\in\mathcal{C}(M, E)\;|\; h=\operatorname{id}\; \text{on}\; \partial M\Big\}.\]
    \end{theorem}
    
	In this paper, we deal with the existence of a nontrivial solution to the following class of Hamiltonian system of the form
	\begin{equation}
		\label{eq 1}
		\begin{cases}
			\Delta u -u = H_u(x,u,v),\; x \in \mathbb{R}^2,\\
			-\Delta v +v= H_v(x,u,v), \;x \in \mathbb{R}^2, \\
			u,v \in H^{1}(\mathbb{R}^2).
		\end{cases}
	\end{equation}
	We are interested in the subcritical growth (at $+\infty$) case, that is, for each $\beta>0$, there exists a constant $K>0$ such that for all $(x,u,v) \in \mathbb{R}^2\times [0,\infty)\times[0,\infty)$, we have
			\begin{equation}
			\label{eq 2}
            H_u(x,u,v) \le Ke^{\beta u^2}\;\;\text{and} \;\;H_v(x,u,v) \le Ke^{\beta v^2},
		\end{equation}
 where $H_w$ designates the partial derivative of $H$ with respect to $w$.
	
	If, we define in $H^{1}(\mathbb{R}^2)\times H^{1}(\mathbb{R}^2)$ the functional
	\begin{equation}
		\label{eq 3}
		J(u,v)= \int_{\mathbb{R}^2} \dfrac{1}{2}\Bigg(\Big( |\Delta v(x)|^2 +v^2(x)\Big) -\Big(|\Delta u(x)|^2 +u^2(x)\Big) \Bigg)dx -\int_{\mathbb{R}^2} \Big(H(x,u,v)\Big)dx,
	\end{equation}
	then, it is well known that the first part of the functional $J$ is strongly indefinite operator and under suitable assumptions, $J$ is of class $\mathcal{C}^1$ on $H^{1}(\mathbb{R}^2)\times H^{1}(\mathbb{R}^2)$, and its critical points are weak solutions of $(\ref{eq 1})$ . Precisely, the functional $J$ has the form
	\begin{equation*}
		\dfrac{\|Q(u,v)\|^2}{2}- \dfrac{\|P(u,v)\|^2}{2} -\varphi(x,u,v), 	\end{equation*}
	for every $(u,v)\in \Big( X= Y\oplus Z:= H^{1}(\mathbb{R}^2)\times H^{1}(\mathbb{R}^2), \|\cdot \|\Big)$, where $Y$ and $Z$ are infinite-dimensional subspaces of $X$, $P$ and $Q$ are, respectively, the orthogonal projections of $X$ into $Y$ and $Z$.\\

    There exist many papers in the literature devoted to the study of the existence of solutions to the semilinear elliptic systems, both in the whole space and in bounded or unbounded domains; see for example \cite{ FDR, SC, CF, ZL, FY, LY, BF, BC, Songo} and the references therein. The sub-critical case was treated for instance in \cite{FDR, ZL, Sodo}.\\
	
		To apply Theorem $\ref{theo 1.1}$, it is crucial that $\dim Y < +\infty$. However, since the functional $J$ given in $(\ref{eq 3})$ has a strongly indefinite quadratic part, the classical generalized Mountain Pass Theorem (Theorem $\ref{theo 1.1}$) cannot be applied to obtain a solution of $(\ref{eq 1})$. Therefore, it is of particular interest to study such a strongly indefinite case.\\

        In \cite{FDR}, using Theorem $\ref{theo 1.1}$ and  Galerkin approximation procedure, the authors obtained the existence of nontrivial solutions  for the following system of two coupled semilinear Poisson equations:
        \begin{equation*}
		(S)
		\begin{cases}
			-\Delta u = g(v),\; v>0,\; x \in \Omega,\\
			-\Delta v = f(u), \; u>0,\;x \in \Omega, \\
			u,\,v \in H^{1}_0(\Omega),
		\end{cases}
	\end{equation*}
where $\Omega$ is a bounded domain in $\mathbb{R}^2$ with smooth boundary $\partial \Omega$, and the nonlinearities $f$ and $g$ have the maximal growth.\\

     We would like to point out that there are various infinite-dimensional generalizations of the Theorem $\ref{theo 1.1}$ in the literature that can help find a nontrivial solution to problem $(\ref{eq 1})$. Among others, we can mention \cite{KS, BB, CW}.\\

       In this paper, we propose a new way to generalize Theorem $\ref{theo 1.1}$ and demonstrate its applicability to strongly indefinite semilinear elliptic systems $(\ref{eq 1})$ with sub-critical linear nonlinearities. To do so, we consider on the Hilbert space $X$ the $\tau-$topology introduced by Kryszewski and Szulkin \cite{KS}. We replace the set $\Gamma$ in Theorem $\ref{theo 1.1}$ by the following noempty set\\
         \begin{center}
		$\Gamma : = \Big\{\gamma : M \rightarrow X \;\Big|\;\gamma \;\; \text{is}\;\; \tau-$continuous, $\gamma \big|_{ \partial M} =\operatorname{id}$ and\\ every  $u\in \operatorname{int}(M)$ has a $\tau-$neighborhood $N_u$ in $X$ such that $(\operatorname{id}-\gamma)(N_u\cap \operatorname{int}(M))$ \\ is contained in a finite-dimensional subspace of $X \Big \}$,
			\end{center}
       where $\operatorname{int}(M)$ is the interior of $ M:= \{ u\in Y \;|\; \|u\|\le R\}\oplus \{ re\;|\; 0\le r\le R\}$ with $R>\rho$ and $e\in Z$ such that $\|e\|=1$.
In addition, we suppose that $J$ is $C^1-$functional such that $J$ is $\tau-$upper semi-continuous, $\nabla J$ is weakly sequentially continuous and $J$ fulfills the linking geometry
\begin{equation*}
				\underset{u\in \partial B_\rho\cap Z}{\inf} \; J(u) \ge \alpha > 0\ge \underset{u \in \partial M}{\sup}\; J(u) . 
			\end{equation*}
Then, we show that there is a Palais-Smale sequence at level $c :=\underset{\gamma \in \Gamma}{\inf} \; \underset{u\in M}{\sup}\; J(\gamma(u))$. It follows that if the functional $J$ satisfies the $(PS)_c$ (see Definition $\ref{def 2.3}\; (3)$), then $c$ is a critical value of $J$; making our infinite-dimensional result a natural way to generalize Theorem $\ref{theo 1.1}$ (see Theorem $\ref{theorem 3.1}$).	The main difficulty here lies in the lack of compactness due to
the unboundedness of the domain, as well as the fact that the nonlinear terms $H_u(x,u,v)$ and $H_v(x,u,v)$ exhibit subcritical exponential growth in the sense of the Trudinger–Moser inequality. By applying our main result (see Theorem $\ref{theorem 3.1}$) together with the Trudinger–Moser inequality \cite{OS}, we prove that problem $(\ref{eq 1})$ admits a nontrivial solution.      
The proof of our abstract result is based on an infinite-dimensional general minimax result proved by the two authors in \cite{CS} (see \cite[Theorem 2.1]{CS}).\\

   By using our abstract result, the use of any Galerkin reduction method is no longer required to obtain a nontrivial solution of problem $(S)$.\\
    
	To the best of our knowledge, there are no results in the literature establishing the existence of solutions to the semilinear elliptic systems of the form $(\ref{eq 1})$ in the whole space.
    
     We believe that our abstract result can be applied to a broader class of indefinite functionals,
	in particular to strongly indefinite ones; that is, functionals of the form
	\begin{equation*}
		J(u) = \dfrac{1}{2} \langle Lu, u \rangle - \Psi(u)
	\end{equation*}
	defined on a Hilbert space $X$, where $L : X \to X$ is a self-adjoint operator whose negative and positive eigenspaces are both infinite-dimensional.\\

	The paper is organized as follows. In the next Section, we recall some classical results and present additional preliminary results which will be used later. In Section 3, we state and prove our main result, and in Section 4 we apply it to obtain a nontrivial solution of $(\ref{eq 1})$.
	
		\section{Kryszewski-Szulkin degree theory and a general minimax principle}
	\subsection{Kryszewski-Szulkin degree theory} In this subsection, we are following the presentation of the degree theory of Kryszewski and Szulkin given in \cite{Wi} by Michel Willem.
	
	Let $Y$ be a real separable Hilbert space endowed with inner product $( \cdot, \cdot )$ and the associated norm $\|\cdot\|$.  
	
	On $Y $ we consider the $\sigma-$topology introduced by Kryszewski and Szulkin; that is, the topology generated by the norm 
	\begin{equation}
    \label{eq 4}
		|u|_\sigma := \sum_{k =0}^{\infty} \frac{1}{2^ {k+1}}|(u, e_k)|, \;\; u \in Y,
	\end{equation}
	where $(e_k)_{k\ge 0}$ is a total orthonormal sequence in $Y$.\\
	
	\begin{remark}	
		\label{rm 1}
		By the Cauchy-Schwarz inequality, one can show that $|u|_\sigma \le \|u\|$ for every $u \in Y$. Moreover, if $(u_n)$ is a bounded sequence in $Y$ then 
		\begin{equation*}
			u_n \rightharpoonup u \Longleftrightarrow u_n \overset{\sigma}{\rightarrow} u,
		\end{equation*}
		where $\rightharpoonup$ denotes the weak convergence and $ \overset{\sigma}{\rightarrow}$ denotes the convergence in the $\sigma-$topology.
	\end{remark}
	Let $U$ be an open bounded subset of $Y$ such that its closure $\overline{U}$ is $\sigma-$closed.
	\begin{definition}[\cite{Wi}]
		\label{definition 2.1}
		A map $f : \overline{U} \rightarrow Y$ is $\sigma-$admissible (admissible for short) if\\
		(1) $f$ is $\sigma-$continuous, \\
		(2) each point $u \in U$ has a $\sigma-$neighborhood $N_u$ in $Y$ such that $(\operatorname{id}-f)(N_u\cap U)$ is contained in a finite-dimensional subspace of $Y$.
	\end{definition}
	\begin{definition}[\cite{Wi}]
    \label{def 2.2}
		A map $h :[0,1]\times \overline{U}\rightarrow Y$ is an admissible homotopy if \\
		(1)  $0 \notin h([0,1] \times \partial U)$,\\
		(2) $h$ is $\sigma-$continuous, that is $t_n \rightarrow t$ and $u_n \overset{\sigma}{\rightarrow} u$ implies $h(t_n,u_n) \overset{\sigma}{\rightarrow} h(t,u)$,\\
		(3) $h$ is $\sigma-$locally finite-dimensional. That is, for any $(t,u) \in[0,1]\times U$ there is a neighborhood $N_{(t,u)}$ in the product topology of $[0,1]$ and $(X, \sigma)$ such that \[\Big\{ v-h(s,v) \; | \; (s,v) \in N_{(t,u)} \cap \Big([0,1]\times U\Big) \Big\}\] is contained in a finite-dimensional subspace of $Y$.
	\end{definition} Then, for an admissible homotopy, by homotopy invariance (see \cite[Theorem 6.6]{Wi}), we have
	\begin{equation*}
		deg \Big(h(0,.), U\Big) = deg\Big(h(1,.),U\Big),
	\end{equation*}
	where $deg$ is the topological degree of $f$ (about 0). Such a degree is called Kryszewski and Szulkin degree (see \cite[Definition 6.3]{Wi}). By existence property (see \cite[Theorem 6.6]{Wi}), if $f : \overline{U} \rightarrow Y$ is admissible with $0 \notin f(\partial U)$ and $deg\Big(f, U\Big)\ne 0$, then there exists $u \in U$ such that $f(u)= 0$.\\
	
	Now, let $X = X^{-}\oplus X^{+}$, where $X^-$ is closed and $X^+=(X^-)^\perp$, be a real separable Hilbert space endowed with the inner product $\langle \cdot, \cdot \rangle$ and the associated norm $\|\cdot\|$. Let $(e_k)_{k\ge 0}$ be an orthogonal basis of $X^-$. We set on $X$, a new norm defined by
	\begin{equation*}
		|u|_\tau := \max \Bigg(\sum_{k =0}^{\infty} \frac{1}{2^ {k+1}}|\langle Pu, e_k\rangle|, \|Qu\| \Bigg), \;\; u \in X.
	\end{equation*}
	Here, $P$ and $Q$ denote the orthogonal projections of $X$ onto $X^-$ and $X^+$, respectively; $\tau$ denotes the topology generated by this norm. The topology $\tau$ was introduced by Kryszewski and Szulkin \cite{KS}. \\
    From $(\ref{eq 4})$, we can write
    \begin{equation*}
        |u|_\tau = \max \Big( |Pu|_\sigma, \|Qu\|\Big), \;\; u\in X.
    \end{equation*}
	\begin{remark}	
		\label{rm 2} For every $u\in X$, 
		we have $\|Qu\| \le \|u\|_\tau$ and $|Pu|_{\sigma} \le \|u\|_\tau$. Moreover, if $(u_n)$ is a bounded sequence in $X$ then 
		\begin{equation}
			\label{eq 5}
			u_n \overset{\tau}{\rightarrow} u \Longleftrightarrow	Pu_n \rightharpoonup Pu \;\;\text{and}\;\; Qu_n \to Qu,
		\end{equation}
		where $\rightharpoonup$ denotes the weak convergence and $ \overset{\tau}{\rightarrow}$ denotes the convergence in the $\tau-$topology.
	\end{remark}
	
	\begin{definition}
    \label{def 2.3}
		Let $J \in \mathcal{C}^1(X,\mathbb{R})$. \\
		\begin{enumerate}
			\item We say that the functional $J$ is $\tau-$upper semi-continuous if  $u_n \overset{\tau}{\rightarrow} u$ implies
			\begin{equation*}
				J(u)\ge \overline{\underset{n\to \infty}{\lim}} J(u_n),
			\end{equation*}
            or, equivalently, for every $c\in \mathbb{R}$, the set
		\begin{equation*}
			\Big\{u\in X\,|\, J(u)\ge c\Big\}
		\end{equation*}
		is $\tau-$closed.
			\item  We say that $\nabla J $ is weakly sequentially continuous if \[u_n \rightharpoonup u \;\; \text{implies}\;\; \nabla J(u_n) \rightharpoonup \nabla J(u).\]
			\item 
			The functional $J$ is said to satisfy the $(PS)_c$ condition (or the Palais-Smale condition at level $c$) if any sequence $(u_n)\subset X$ such that 
			\begin{equation*}
				J(u_n) \rightarrow c \;\; \textit{and} \;\; J' (u_n) \rightarrow 0, \;\; \text{as}\;\;n\to \infty,
			\end{equation*}
			has a convergent subsequence.\\
		\end{enumerate}
	\end{definition}
	
	\subsection{A general minimax principle}
	We consider the class of $\mathcal{C}^1-$functionals $J : X \rightarrow \mathbb{R}$ such that 
	
	\text{(A)}	$J$ is $\tau-$upper semi-continuous and $\nabla J$ is weakly sequentially continuous.\\

	We recall the following two results from \cite{CS} that will play a key role in the proof of our abstract result; see \cite[Theorem 2.1 ]{CS} and \cite[Corollary 3.1]{CS}.
	
	\begin{theorem}[General minimax principle]
		\label{theorem 2.1}
		Assume that $J\in \mathcal{C}^1(X, \mathbb{R})$ satisfies $(A)$, that is, $J$ is $\tau-$upper semi-continuous and $\nabla J$ is weakly sequentially continuous. Let $M$ be a closed subset of $X$, and let $M_0$ be a closed subset of $M$. Let $\operatorname{int}(M)$ be the interior of $M$.
		Let us define 
		\begin{center}
			$\Lambda_0 : = \Big\{\gamma_0 : M_0 \rightarrow X \;\Big|\;\gamma_0 $\;\; \text{is}\;\; $\tau-$continuous $\Big\}$,
		\end{center}
		\begin{center}
			$\Gamma : = \Big\{\gamma : M \rightarrow X \;\Big|\;\gamma \;\; \text{is}\;\; \tau-$continuous, $\gamma \big|_{ M_0} \in \Lambda_0$ and\\ every  $u\in \operatorname{int}(M)$ has a $\tau-$neighborhood $N_u$ in $X$ such that $(\operatorname{id}-\gamma)\Big(N_u\cap\; \operatorname{int}(M)\Big)$ \\ is contained in a finite-dimensional subspace of $X \Big \}$.
		\end{center}
		If $J$ satisfies
		\begin{equation}
			\label{eq 6}
			\infty > c := \underset{\gamma \in \Gamma}{\inf}\; \underset{u \in M}{\sup} \;J(\gamma(u))> a := \underset{\gamma_0 \in \Lambda_0}{\sup}\; \underset{u \in M_0}{\sup} \; J(\gamma_0(u)),
		\end{equation}
		Then,
		for every $\varepsilon \in \big( 0, \frac{c-a}{2}\big)$, $ \delta >0$ and $\gamma \in \Gamma$ such that
		\begin{equation}
			\label{eq 7}
			\underset{u\in M}{\sup}\; J (\gamma (u)) \le c +\varepsilon , 
		\end{equation}
		there exists $u \in X$ such that 
		\begin{enumerate}
		\item[$(a)$]\[c-2\varepsilon \le J(u) \le c+2\varepsilon,\]
		\item[$(b)$]\[dist(u, \gamma(M)) \le 2 \delta,\]
		\item[$(c)$]\[\|J' (u)\| < \frac{8\varepsilon}{\delta}.\]
		\end{enumerate}
	\end{theorem}
	\begin{corollary}
		\label{cor 2.1}
		Suppose that the assumptions of Theorem $\ref{theorem 2.1}$ are satisfied and suppose that $J$ satisfies $(\ref{eq 6})$. Then, there exists a sequence $(u_n)\subset X$ satisfying 
		\begin{equation*}
			J(u_n) \rightarrow c, \;\; J'(u_n) \rightarrow 0, \;\;\text{as}\;\; n\to \infty.
		\end{equation*}
        
		In particular, if $J$ satisfies $(PS)_c$ condition, then $c$ is a critical value $J$.
	\end{corollary}
	
	\section{A new generalized Mountain Pass Theorem}
	In this section, we state and prove the main result of this paper, which is an infinite-dimensional generalization of the classical Mountain Pass Theorem of P. H. Rabinowitz \cite[Theorem 5.3]{Ra}.\\

    Let $M$ be a bounded closed subset of $X$, and let $\operatorname{int}(M)$ denote the interior of $M$.
    \begin{definition}
\label{def 3.1}
       A map $\gamma : M \rightarrow X$ is $\tau-$admissible if $\gamma \big|_{ \partial M} = \operatorname{id}$ and
       \begin{enumerate}
           \item[$(i)$] $\gamma$ is $\tau-$continuous,
           \item[$(ii)$] every  $u\in \operatorname{int}(M)$ has a $\tau-$neighborhood $N_u$ in $X$ such that $(\operatorname{id}-\gamma)\Big(N_u\cap\; \operatorname{int}(M)\Big)$ \\ is contained in a finite-dimensional subspace of $X$. 
       \end{enumerate}
    \end{definition}	
	\begin{remark} We have:
    \label{rm 3}
    \begin{enumerate}
        \item[$(i)$] The set of all $\tau-$admissible maps is nonempty.
        \item[$(ii)$] Every $\tau-$admissible map is continuous: indeed, let $(u_n)\subset M$ such that $u_n\to u\in M$. Then, by Remark $\ref{rm 2}$, $u_n\overset{\tau}{\to}u$ and $\gamma(u_n)\overset{\tau}{\to} \gamma(u)$. It follows that $\gamma (u_n)\to \gamma(u)$ because for every sufficiently large $n$, both $\gamma(u_n)$ and $\gamma(u)$ belong to a finite-dimensional subspace of $X$ on which the $\tau$ topology coincides with the norm topology.
    \end{enumerate}
	   
	\end{remark}
Let $\zeta>0$.	Let $B_\zeta$ denote the open ball in $X$ of radius $\zeta$ centered at $0$, and let $\partial B_\zeta$ denote its boundary.
	\subsection{The main result}
	\begin{theorem}[Generalized Mountain Pass Theorem]
		\label{theorem 3.1}
		Let $\Big(X = X^{-}\oplus X^{+}, \|\cdot\|\Big)$ be a real separable Hilbert space with $X^-$ a closed separable subspace of $X$ which could be infinite-dimensional and $X^+=(X^-)^\perp$. Assume that $J\in \mathcal{C}^1(X,\mathbb{R})$ satisfies $(A)$. Suppose that
        \begin{enumerate}
            \item[$(J_1)$] there are constants $\rho$, $\alpha >0$ such that
		\begin{equation}
			\label{eq 8}
			b:= \underset{u\in \partial B_\rho \cap X^+}{\inf} \; J(u) \ge \alpha, \;\; \text{and} 
		\end{equation}
        \item[$(J_2)$] there is an $e \in \partial B_1\cap X^+$ and $R> \rho$ such that if 
		\begin{equation*}
			M:= \Big\{u\in X^-\;|\; \|u\|\le R\Big\}\oplus \Big\{re\;|\;0\le r\le R\Big\},
		\end{equation*}
		then we have
		\begin{equation}
			\label{eq 9}
			\underset{u\in \partial M}{\sup}\; J(u)\le 0.
		\end{equation}
        \end{enumerate}
        
		Let $c \in \mathbb{R}$ be characterized as 
        \[c:= \underset{\gamma \in \Gamma}{\inf} \; \underset{u\in M}{\sup}\; J(\gamma(u)),\] 
			where $\Gamma$ is the set of all $\tau-$admissible maps in Definition $\ref{def 3.1}$.
            
	Then, there exists a Palais--Smale sequence at level $c$ in $X$. That is,  a sequence $(u_n)\subset X$ such that 
		\begin{equation*}
			J(u_n) \rightarrow c \ge \alpha, \;\; J'(u_n) \rightarrow 0, \;\; \text{as}\;\; n\to \infty.
		\end{equation*}
	\end{theorem} 
\begin{remark} Observe that:
    \begin{enumerate}
        \item[$(i)$] If the functional $J$ in Theorem $\ref{theorem 3.1}$ satisfies the $(PS)_c$ condition, then \[c= \underset{\gamma \in \Gamma}{\inf} \; \underset{u\in M}{\sup}\; J(\gamma(u))\] is a critical value of $J$ by Corollary $\ref{cor 2.1}$. Here, in Theorem $\ref{theorem 2.1}$, $M_0 =\partial M$ and $\Lambda_0=\{\operatorname{id}\}$.
        \item [$(ii)$] By Reamrk $\ref{rm 3}\,(ii)$ and the fact in finite-dimensional spaces, all norms are equivalent, if $\dim X^- < \infty$, then the conclusion of Theorem $\ref{theorem 3.1}$ remains valid for any functional $J\in \mathcal{C}^1(X, \mathbb{R})$: this is the classical Mountain Pass Theorem $\ref{theo 1.1}$.
    \end{enumerate}
    Therefore, we will only prove Theorem $\ref{theorem 3.1}$ for the case $\dim X^- =\infty$.
\end{remark}
\subsection{Proof of Theorem $\ref{theorem 3.1}$}
Here we prove Theorem $\ref{theorem 3.1}$.
\begin{proof}
The set $M$ is $\sigma-$closed, that is, if $\Big\{u_n =y_n +r_ne\Big\}_{n\in \mathbb{N}}\subset M$ such that $y_n \overset{\sigma}{ \rightarrow } y$ in $X^-$ with $\|y_n\|\le R$ and $r_n \to r$ in $[0,R]$, then $u=y+re \in M$. Indeed, since the set $ \{x\in X^- \;|\; \|x\| \le R\}$ is bounded, by Remark $\ref{rm 1}$,  $y_n \rightharpoonup y$  and $\|y\| \le \underset{n \rightarrow \infty}{\lim}\inf\; \|y_n\|\le R$. Since $0\le r\le R$, we conclude that $u=y+re \in M$.

In order to apply Corollary $\ref{cor 2.1}$, we have only to show that $c\ge \alpha$. To this end, we will prove that, for every $\gamma \in \Gamma$, 
		\begin{equation*}
			\gamma (M) \cap(\partial B_\rho \cap X^+)\ne \emptyset.
		\end{equation*}
        
		Let $P$ denote the projection of $X$ onto $X^-$ such that $PX^+=\{0\}$.
		Let $\gamma \in \Gamma$ and consider the map 
		\begin{eqnarray*}
			H &:& [0,1]\times M \rightarrow X^-\oplus \mathbb{R}e\\  (t,u) &\mapsto& tP(\gamma(u)) +(1-t)y + \Big(t\|\gamma(u)-P\gamma(u)\|+(1-t)r -\rho\Big)e.
		\end{eqnarray*}
		\begin{enumerate}
			\item[(i)] Since $\gamma$ is $\tau-$continuous, the map $H$ is also $\tau-$continuous by construction. 
			\item[(ii)] Let $(t,u)\in[0,1]\times \partial M $. Then $u=y+re$ with $r\in \{0, R\}$ and $\|y\|\le R$ or $0\le r\le R$ and $\|y\|=R$. Since $\gamma =\operatorname{id}$ on $\partial M$ and $Pu =P(y+re) =Py =y$, then 
			\begin{equation*}
				H(t,u) = y+\Big(r-\rho\Big)e \ne 0.
			\end{equation*}
			\item[(iii)] Let $u\in \operatorname{int}(M)$. Then $u$ has a $\tau-$neighborhood $N_u$ in $X$ such that \[(\operatorname{id}-\gamma)\Big(N_u\cap \;\operatorname{int}(M)\Big) \subset E_0,\] where $E_0$ is a finite-dimensional subspace of $X$. Let $v\in N_u\cap\: \operatorname{int}(M)$.\\ We have 
			\begin{equation*}
				v -H(t,v)= tP(v-\gamma(v)) + \Big(-t\|\gamma(v)-P\gamma(v)\| +tr +\rho\Big)e \in E_0\oplus \mathbb{R}e.
			\end{equation*}
		\end{enumerate}
		Since the topology $\sigma$ is induced by the topology $\tau$ on $X^-$, we conclude  that the map 
		\begin{equation*}
			H : [0,1]\times M \rightarrow X^-\oplus \mathbb{R}e
		\end{equation*}
		is  $\sigma-$continuous and $\sigma-$locally finite-dimensional and $0\notin H([0,1]\times \partial M)$.\\ Therefore, $H$ is an admissible homotopy (see Definition $\ref{def 2.2}$), and the Kryszewski-Szulkin degree 
		\begin{equation*}
			deg\Big(H(t,\cdot), \operatorname{int}(M)\Big)
		\end{equation*}
		is independent of $t\in [0,1]$. Hence, by homotopy invariance of Kryszewski-Szulkin degree, we have
		\begin{equation*}
			deg \Big(H(1,\cdot), \operatorname{int}(M)\Big) = deg \Big(H(0,\cdot), \operatorname{int}(M)\Big) = deg\Big(\operatorname{id}-\rho e, \operatorname{int}(M)\Big), 
		\end{equation*}
		where 
		\begin{equation*}
			H(1,u)= P\gamma(u)+\Big(\|\gamma(u)-P\gamma(u)\|-\rho\Big)e, \;\; \text{and}\;\; H(0,u) = u-\rho e.
		\end{equation*}
		Because $0<\rho<R$, then $\rho e \in \operatorname{int}(M)$. By normalization property of Kryszewski-Szulkin degree, we have
		\begin{equation*}
			deg \Big(H(1,\cdot), \operatorname{int}(M)\Big) = deg\Big(\operatorname{id}-\rho e, \operatorname{int}(M)\Big) =1.
		\end{equation*}
		By existence property of Kryszewski-Szulkin degree, there exits $ x \in \operatorname{int}(M)$ such that $H(1,x) =0$. This implies that \[P\gamma(x)=0\;\; \text{and} \;\; \|\gamma(x)\|=\rho.\] Consequently, for each $\gamma \in \Gamma$, there exists $x=x(\gamma) \in \operatorname{int}(M)$ such that
		\begin{equation*}
			\gamma(x)= (\operatorname{id}-P)\gamma (x) \in X^+\;\; \text{and}\;\; \|\gamma(x)\| =\rho.
		\end{equation*}
		It follows that for every $\gamma \in \Gamma$, 
		\begin{equation*}
			\alpha\le	b=	\underset{z \in \partial B_\rho\cap X^+}{\inf} J (z) \le J(\gamma(x)) \le \underset{u \in M}{\sup}\; J(\gamma(u)).
		\end{equation*}
		Hence, $c\ge \alpha$ and applying Corollary $\ref{cor 2.1}$, we obtain the existence of $(u_n)\subset X$ such that 
		\begin{equation*}
			J(u_n) \rightarrow c, \;\; J'(u_n) \rightarrow 0, \;\; \text{as}\;\; n\to \infty.
		\end{equation*} The proof of Theorem $\ref{theorem 3.1}$ is thus complete.
        \end{proof}

	\section{Application}
		In this section, we demonstrate how our main result, Theorem $\ref{theorem 3.1}$, can be employed to establish an existence theorem for nonlinear elliptic systems, in particular for problem $(\ref{eq 1})$.
	\subsection{An existence result}
	Consider the Hilbert space 
	\begin{equation*}
		H^1(\mathbb{R}^2) := \Big\{u\in L^2(\mathbb{R}^2)\;|\; |\nabla u |\in L^2(\mathbb{R}^2)\Big\},
	\end{equation*}
	endowed with the inner product 
	\begin{equation*}
		(u,v)_1 = \int_{\mathbb{R}^2} \Big(\nabla u \nabla v + uv\Big) dx,
	\end{equation*}
	and the associated norm $\|.\|_1$ given by
	\begin{equation*}
		\|u\|_{1}^2 :=	\int_{\mathbb{R}^2} \Big(|\nabla u |^2 + |u|^2\Big) dx.
	\end{equation*}
	
	We denote the product space $X= H^{1}(\mathbb{R}^2)\times H^{1}(\mathbb{R}^2)$, the Hilbert space endowed with the inner product 
	\begin{equation}
		\langle (u,v), (\phi,\psi)\rangle := \int_{\mathbb{R}^2} \Big( \nabla u \nabla \phi + u\phi \Big) dx + \int_{\mathbb{R}^2} \Big(\nabla v  \nabla \psi +v\psi \Big)dx, \;\; \text{for every}\;\; (\phi, \psi)\in X,
	\end{equation}
	and the corresponding norm $\|\cdot\|$ given by
	\begin{equation*}
		\|(u,v),(u,v)\|^2 = \| u \|_{1}^2 + \|v\|_{1}^2.
	\end{equation*}
    
	If we define
	\begin{equation*}
		X^- := \Big\{(u,0) \in X \Big\} \;\; \text{and}\;\;  X^+:=\Big\{(0,v)\in X\Big\},
	\end{equation*}
    
	since we can write $(u,v)$ as
	\begin{equation*}
		\Big(u,v\Big) = \Big(u,0\Big)+\Big(0, v\Big),
	\end{equation*}
	then, $X= X^-\oplus X^+$.
	
	Let us denote by $P$ the projection of $X$ onto $X^-$ and by $Q$ the projection of $X$ onto $X^+$. We have
	\begin{equation*}
	\int_{\mathbb{R}^2} \dfrac{1}{2}\Bigg(\Big( |\nabla v|^2 +v^2\Big) -\Big(|\nabla u|^2 +u^2\Big) \Bigg)dx =  \dfrac{\|Q(u,v)\|^2}{2} -\dfrac{\|P(u,v)\|^2}{2}.
	\end{equation*}
	
	Finally, let us define the functional $J: X \rightarrow \mathbb{R}$ given by
	\begin{align}
		\label{eq 11}
		J(u,v)&= \int_{\mathbb{R}^2} \dfrac{1}{2}\Bigg(\Big( |\nabla v|^2 +v^2\Big) -\Big(|\nabla u|^2 +u^2\Big) \Bigg)dx -\int_{\mathbb{R}^2} \Big(H(x,u,v)\Big)dx \nonumber\\
		&= \dfrac{\|Q(u,v)\|^2}{2} -\dfrac{\|P(u,v)\|^2}{2} -\varphi(,u,v) \nonumber\\
		&= \dfrac{1}{2}\|(0,v)\|^2-  \dfrac{1}{2}\|(u,0)\|^2- \varphi(u,v),
	\end{align}
	where
	\begin{equation*} 
		\varphi(u,v) :=  \int_{\mathbb{R}^2} \Big(H(x,u,v)\Big) dx.\\
	\end{equation*}
		\begin{definition}
		\label{def 4.1}
		We say that $(u,v)$ is a weak solution of problem $(\ref{eq 1})$ if $(u,v)\in X$, and satisfies for any $ (w,z) \in X$: 
		\begin{equation*}
			\int_{\mathbb{R}^2} \Big( \nabla v \nabla z + vz \Big) dx - \int_{\mathbb{R}^2} \Big( \nabla u \nabla w + uw \Big) dx  -\int_{\mathbb{R}^2} \Big( wH_u(x,u,v) +zH_v(x,u,v) \Big)dx =0.\\
		\end{equation*}
	\end{definition}
	
	Throughout this section, $|\cdot|_p$ stands for the $L^p$-norm; $\to$ and $\rightharpoonup$ denote strong and weak convergence, respectively.\\
	
	Our assumptions on $(\ref{eq 1})$ are the following: 
	
	\begin{enumerate}
		\item[$(H_1)$] The function $H \in \mathcal{C}^1\Big(\mathbb{R}^2\times\mathbb{R}\times\mathbb{R}, \mathbb{R}\Big)$ and $H(x,0,0)=0$ for every $x\in \mathbb{R}^2$. 
		\item[$(H_2)$] The nonlinearities $H_u$ and $H_v$ satisfy $(\ref{eq 2})$ and 
		\begin{equation*}
			\lim\limits_{u\to 0}\dfrac{H_u(x,u,v)}{ u}= \lim\limits_{v\to 0}\dfrac{H_v(x,u,v)}{ v}=0, \;\; \text{uniformly with respct to}\;\; x\in \mathbb{R}^2.
		\end{equation*}
			
		\item[$(H_3)$] There exists $\mu>2$ such that for all $(u,v)\ne (0,0)$ and for every $x\in \mathbb{R}^2$, we have 
		\begin{equation*}
			0< \mu H(x,u,v)\le uH_u(x,u,v)+ vH_v(x,u,v),
		\end{equation*}
        with \[uH_u(x,u,v)>0, \;\; vH_v(x,u,v)>0.\]
	\end{enumerate}
    \begin{example}
       Let $x\in \mathbb{R}^2$. Let $r>2$.  Typical examples of functions satisfying the assumptions $(H_1)$, $(H_2)$ and $(H_3)$ are
       \begin{equation*}
       H(x,\theta, \zeta)= e^{-|x|^2}\dfrac{|\theta|^r+|\zeta|^r}{r}, \;\; \text{for every}\;\; (x,\theta,\zeta)\in \mathbb{R}^2\times\mathbb{R}\times\mathbb{R},
       \end{equation*}
       and 
       \begin{equation*}
           H(x,\theta, \zeta)= e^{-|x|^2}\Big[ \Big( |\theta|^3 e^{|\theta|}\Big) + \Big( |\zeta|^3 e^{|\zeta|}\Big)\Big], \;\; \text{for every}\;\; (x,\theta,\zeta)\in \mathbb{R}^2\times\mathbb{R}\times\mathbb{R}.
       \end{equation*}
      \end{example}
		\begin{remark}
		Since assumption $(H_1)$ implies that $H(x,0,0)=0$ for every $x\in \mathbb{R}^2$, then $(u,v)=(0,0)$ is a trivial solution of problem $(\ref{eq 1})$.
	\end{remark}
	
	Here is the main result of this section:
	
	\begin{theorem}
		\label{theo 4.1}
	Under assumptions $(H_1)$, $(H_2)$ and $(H_3)$, problem $(\ref{eq 1})$ admits at least one nontrivial solution.
	\end{theorem}

	To facilitate the proof of Theorem $\ref{theo 4.1}$, we first present a few lemmas. \\
	
	In 2001, do Ó and Souto \cite{OS} proved in the whole space $\mathbb{R}^2$ a version of the Trudinger-Moser inequality, namely, 
	\begin{enumerate}
		\item  if  $u\in H^1(\mathbb{R}^2)$ and  $\beta>0$, then
		\begin{equation}
			\label{eq 12}
			\int_{\mathbb{R}^2}\big(e^{\beta |u|^2}-1\big)dx < + \infty;
		\end{equation}
		\item if $0< \beta < 4\pi$ and $|u|_{L^2(\mathbb{R}^2)} \le m_0 $, then there exists a constant $n_0$, which depends only on $\beta$ and $m_0$, such that 
		\begin{equation*}
			\underset{|\nabla u|_{L^2(\mathbb{R}^2)} \le 1}{\sup} \int_{\mathbb{R}^2}\big(e^{\beta |u|^2}-1\big)dx \le n_0.\\
		\end{equation*}
	\end{enumerate}
	For the proof of the following lemma, see \cite{SC2}.
	
	\begin{lemma}
		\label{lem 4.1}
	Under assumption $(H_1)$, let the nonlinearities $H_u$ and $H_v$ satisfy assumption $(H_2)$. Then, for given $\varepsilon >0$, there exists $ C = C(\varepsilon)  >0$ such that, for every $(x,u,v)\in \mathbb{R}^2\times\mathbb{R}\times\mathbb{R}$, we have
	\begin{equation*}
		|H_u(x,u,v)| \le \varepsilon |u| +C\Big(e^{\beta u^2}-1\Big)\;\; \text{and}\;\; |H_v(x,u,v)| \le  \varepsilon |v|+C\Big(e^{\beta v^2}-1\Big).
	\end{equation*}
	\end{lemma}
	By the fundamental theorem of calculus and Lemma $\ref{lem 4.1}$, for every $(x,u,v)\in \mathbb{R}^2\times\mathbb{R}\times\mathbb{R}$ and for every $\delta\ge 1$, we have
	\begin{align}
		\label{eq 13}
	|H(x,u,v)|&\le \int_{0}^{1}\Big|\dfrac{d}{dt}H(x,tu,tv)\Big|dt \nonumber\\
	&\le\int_{0}^{1}\Big|uH_u(x,tu,tv) +vH_v(x,tu,tv)\Big|dt\nonumber\\
	&\le \varepsilon(|u|^2+|v|^2) \int_{0}^{1}tdt+ C|u|\int_{0}^{1}\big(e^{\beta |tu|^2}-1\big)dt + C|v|\int_{0}^{1}\big(e^{\beta |tv|^2}-1\big)dt\nonumber \\
	&\le \varepsilon(|u|^2+|v|^2) + C|u|^\delta\big(e^{\beta |u|^2}-1\big) + C|v|^\delta\big(e^{\beta |v|^2}-1\big).
	\end{align}

	By $(\ref{eq 12})$ and \cite[Lemma 2.2]{dMS} and \cite[Remark 2.3]{dMS} (see also \cite[Remark 5]{SC2}), for every $(u,v)\in X$ and for each $q\ge 1$, there is $C=C(\varepsilon, q)>0$  such that \[C|u|^q\big(e^{\beta |u|^2}-1\big)\;\; \text{and}\;\; C|v|^q\big(e^{\beta |v|^2}-1\big) \;\; \text{belong to}\;\; L^1(\mathbb{R}^2).\]
	It follows that, for any $(u,v)\in X$, \[H(x,u,v)\in L^1(\mathbb{R}^2).\]
		Therefore, the functional $J(u,v)$ given in $(\ref{eq 11})$ is well defined. Furthermore, using variational standard arguments, the functional $J(u,v)$ is of class $\mathcal{C}^1$ in $X$ and 
	\begin{multline}
		\label{eq 14}
		J'(u,v)(w, z) = 	\int_{\mathbb{R}^2} \Big( \nabla v \nabla z + vz \Big) dx - \int_{\mathbb{R}^2} \Big( \nabla u \nabla w + uw \Big) dx \\ -\int_{\mathbb{R}^2} \Big( wH_u(x,u,v) +zH_v(x,u,v) \Big)dx, \;\: \text{for every}\;\; (w, z) \in X. 
	\end{multline}
	
	Hence, the weak solutions of problem $(\ref{eq 1})$ are exactly the critical points of $J(u,v)$ in $X$ (see Definition $\ref{def 4.1}$).
    
	We will need the following convergence result due to de  Figueiredo and al. \cite[Lemma 2.1]{FMR}.
    \begin{lemma}
    \label{lem 4.2}
        Let $\Omega \subset \mathbb{R}^2$ be a bounded domain and 
$f : \mathbb{R}^2 \times \mathbb{R} \to \mathbb{R}$ a continuous function. Then for any sequence $(u_m)$ in $L^1(\Omega)$ such that 
$u_m \to u$ in $L^1(\Omega)$,
\[f(x,u_m) \in L^1(\Omega)\;\; \text{and} \;\; \int_{\Omega} \left| f(x,u_m) u_m \right| \, dx \le \overline{C},\] up to a subsequence, we have
\[
f(x,u_m) \to f(x,u) \;\; \text{in }\;\; L^1(\Omega).
\]
    \end{lemma}
		\begin{lemma}
		\label{lem 4.3}
		The functional $J(u,v)$ given in $(\ref{eq 11})$ satisfies $(A)$, that is, $J(u,v)$ is $\tau-$upper semi-continuous and $\nabla J$ is weakly sequentially continuous.
	\end{lemma}
	\begin{proof}
		
		1. The functional $J$ is  $\tau-$upper semi-continuous:
		for every $c\in \mathbb{R}$, let us show that the set
			$\Big\{(u,v)\in X\,|\, J(u,v)\ge c\Big\}$
		is $\tau-$closed.
        
		Let $(u_n,v_n) \subset X$ such that $(u_n,v_n)\overset{\tau}{\rightarrow} (u,v)$ in $X$ and $c\le J(u_n,v_n)$. Then, $(u_n,v_n)$ is bounded. Indeed, by Remark $\ref{rm 2}$, $\Big(\|Q(u_n,v_n)\|\Big)$ is bounded; since 
		\begin{equation*}
			\varphi(u,v) = \int_{\mathbb{R}^2} \Big(H(x,u,v)\Big)dx \ge 0,
		\end{equation*} 
		then $\|P(u_n,v_n)\|^2 = \|Q(u_n,v_n)\|^2 -2 J(u_n,v_n) -2 \varphi(x,u_n,v_n)\le \|Q(u_n,v_n)\|^2 -2c$.
		It follows that $\Big(\|P(u_n,v_n)\|\Big)$ is also bounded. 
        
		Thus, there is a subsequence of $(u_n,v_n)$ that we still denote $(u_n, v_n)$ such that  
		$(u_n,v_n) \rightharpoonup (u,v)$ in $X$. So, up to subsequence, we may assume  \begin{align*}
        (u_n,v_n) &\to (u,v), \;\; \text{in}\;\; L^2_{loc}(\mathbb{R}^2)\times L^2_{loc}(\mathbb{R}^2), \\
			(u_n(x),v_n(x))&\to (u(x),v(x)),\;\; \text{almost everywhere, in}\;\;\mathbb{R}^2.
		\end{align*}
		
		Since $H(x,u_n,v_n)\ge 0$, then by the Fatou Lemma, we have 
		\begin{equation*}
			\int_{\mathbb{R}^2} H(x,u,v) dx= \int_{\mathbb{R}^2} \underset{n\rightarrow \infty}{\underline{\lim}} H(x,u_n,v_n)dx  \le \underset{n\rightarrow \infty}{\underline{\lim}} \int_{\mathbb{R}^2} H(x,u_n,v_n)dx.
		\end{equation*}
        
		Since $\|\cdot\|$ is weak lower semi-continuous, we have 
		\begin{equation*}
			\|P(u,v)\|^2 \le \underset{k\rightarrow \infty}{\underline{\lim}} \|P(u_n,v_n)\|^2.
		\end{equation*}
        
		Moreover, since $\|Q(u_n,v_n)\|^2 \to \|Q(u,v)\|^2$, we have 
		\begin{equation*}
			\underset{k\rightarrow \infty}{\overline{\lim}} \Big( -\|Q(u_n,v_n)\|^2\Big)= 
			-\|Q(u,v)\|^2 = 	\underset{k\rightarrow \infty}{\underline{\lim}} \Big(-\|Q(u_n,v_n)\|^2\Big).
		\end{equation*}
        
		Hence, 
		\begin{align*}
			-J(u,v) &= \dfrac{1}{2}\Big( \|P(u,v)\|^2-\|Q(u,v)\|^2 \Big) +\varphi (u,v)\\ &\le \underset{k\rightarrow \infty}{\underline{\lim}} \Bigg(\dfrac{1}{2}\Big( \|P(u_n,v_n)\|^2-\|Q(u_n,v_n)\|^2 \Big) +\varphi (u_n,v_n)  \Bigg)\\
			&=  \underset{k\rightarrow \infty}{\underline{\lim}} (-J(u_n,v_n) )
			= - \underset{k\rightarrow \infty}{\overline{\lim}} J(u_n,v_n)
			\le -c.
		\end{align*}
	2. Now, let us show that $\nabla J$ is weakly sequentially continuous: let $(u_n,v_n) \subset X$ such that $(u_n,v_n) \rightharpoonup (u,v)$ in $X$. We have to show that $\nabla J(u_n,v_n) \rightharpoonup \nabla J(u,v)$.
    
    We have 
	\begin{equation*}
		\int_{\mathbb{R}^2}\Big(\nabla v_n \nabla z +v_n z -\nabla u_n \nabla w -u_n w\Big)dx \to \int_{\mathbb{R}^2}\Big(\nabla v\nabla z +vz -\nabla u \nabla w -uw\Big)dx.
	\end{equation*}
Let $(w,z)\in \mathcal{C}_0^\infty(\mathbb{R}^2)\times \mathcal{C}_0^\infty(\mathbb{R}^2)$	and let us set $ \Omega:= supp \,\{(w,z)\}\subset \mathbb{R}^2$, the support of $(w,z)$. We have
	\begin{equation*}
		\int_{\mathbb{R}^2}\Big( wH_u(x,u_n,v_n)+zH_v(x,u_n,v_n)\Big)dx = \int_{\Omega}\Big( wH_u(x,u_n,v_n)+zH_v(x,u_n,v_n)\Big)dx.
	\end{equation*}
	Since $(u_n,v_n) \rightharpoonup (u,v) \;\;\text{in} \;\; X$, by compact Sobolev embedding theorem, going if necessary to a subsequence, we have
	\begin{align*}
		(u_n,v_n) &\to (u,v) \;\; \text{in}\;\; L^2(\Omega)\times L^2(\Omega) \subset L^1(\Omega)\times L^1(\Omega), \;\; \text{as}\;\; n \to \infty.
	\end{align*}
	The functions $H_u$ and $H_v$ are continuous by assumption $(H_1)$. 
	Moreover, by Lemma $\ref{lem 4.1}$ and $(\ref{eq 12})$, we have
	\begin{equation*}
		\int_{\Omega}|H_u(x,u_n(x),v_n(x))|\,dx\le \varepsilon \int_{\Omega}\Big(|u_n| +c\Big(e^{\beta (u_n)^2}-1\Big)\Big)\,dx <+\infty,
	\end{equation*}
	\begin{equation*}
	\int_{\Omega}|H_v(x,u_n(x),v_n(x))|\,dx\le \varepsilon \int_{\Omega}\Big(|v_n| +c\Big(e^{\beta (v_n)^2}-1\Big)\Big)\,dx <+\infty.
	\end{equation*}
    In the same way,
    \begin{equation*}
        \int_{\Omega}|u_nH_u(x,u_n(x),v_n(x))|\,dx\le \varepsilon \int_{\Omega}\Big(|u_n|^2 +c|u_n|\Big(e^{\beta (u_n)^2}-1\Big)\Big)\,dx \le \bar{c}_1,
    \end{equation*}
    \begin{equation*}
        \int_{\Omega}|v_nH_u(x,u_n(x),v_n(x))|\,dx\le \varepsilon \int_{\Omega}\Big(|v_n|^2 +c|v_n|\Big(e^{\beta (v_n)^2}-1\Big)\Big)\,dx \le \bar{c}_2,
    \end{equation*}
    for some constants $\bar{c}_1>0$, $\bar{c}_2>0$. For every $(w,z)\in \mathcal{C}_0^\infty(\mathbb{R}^2)\times\mathcal{C}_0^\infty(\mathbb{R}^2)$,
	 Lemma $\ref{lem 4.2}$ implies that
	\begin{equation*}
		\int_{\Omega}\Big( wH_u(x,u_n,v_n)+ zH_v(x,u_n,v_n)\Big)dx \to \int_{\Omega}\Big( wH_u(x,u,v)+zH_v(x,u,v)\Big)dx, \;\; \text{as}\;\: n\to \infty.
	\end{equation*}
	Thus, by $(\ref{eq 14})$, we have 
	\begin{equation*}
		\Big| J'(u_n,v_n) (w,z)- J'(u,v)(w,z)  \Big| \;\; \rightarrow 0, \;\; \text{as}\;\;n\rightarrow \infty.
	\end{equation*}
	Therefore, for every $(w,z)\in \mathcal{C}_0^\infty(\mathbb{R}^N)\times \mathcal{C}_0^\infty(\mathbb{R}^N) $,
	\begin{equation*}
		\langle \nabla J(u_n,v_n), (w,z) \rangle \to \langle \nabla J(u,v), (w,z) \rangle.
	\end{equation*}
	Hence, $\nabla J(u_n,v_n) \rightharpoonup \nabla J(u,v)$.
	\end{proof}
    
    The next growth condition is essentially identical to that considered by several authors (see e.g., \cite{C}). 
	\begin{lemma}
		\label{lem 4.4}
	Under assumptions $(H_1)$ and $(H_3)$, there exists $C_0>0$ such that 
	\begin{equation*}
		H(x,u,v)\ge C_0\Big(|u|^\mu +|v|^\mu\Big), \;\; \text{for every}\;\; (x,u,v)\in \mathbb{R}^2\times \mathbb{R} \times \mathbb{R}.
	\end{equation*}
	\end{lemma}
	The next two results show that the functional $J$ given in $(\ref{eq 11})$ satisfies the geometric assumptions $(\ref{eq 8})$ and $(\ref{eq 9})$ of Theorem $\ref{theorem 3.1}$.

\begin{lemma}
	\label{lem 4.5}
	There exist $\rho, \; \alpha >0$ such that 
	$b:= \underset{(0,v)\in \partial B_\rho \cap X^+}{\inf} \; J(0,v) \ge \alpha$.
\end{lemma}
	\begin{proof}
		On $X^+$, we have 
		\begin{equation*}
			J(0,v)= \dfrac{1}{2}\|v\|_1^2 -\int_{\mathbb{R}^2} H(x,0,v)dx. 
		\end{equation*}
		From $(\ref{eq 13})$, for every $\varepsilon>0$, there exits $C>0$ such that for every $q>2$,
		\begin{equation*}
			|H(x,u,v)|\le \varepsilon(|u|^2+|v|^2) + C|u|^q\big(e^{\beta |u|^2}-1\big) + C|v|^q\big(e^{\beta |v|^2}-1\big),
		\end{equation*}
		which implies that 
		\begin{equation*}
		J(0,v)\ge \dfrac{1}{2}\|v\|_1^2 -\varepsilon\int_{\mathbb{R}^2} |v|^2dx -C\int_{\mathbb{R}^2}|v|^q(e^{\beta |v|^2}-1)dx.
		\end{equation*}
		Again, by \cite[Remark 5]{SC2},
		\begin{equation*}
C\int_{\mathbb{R}^2}|v|^q(e^{\beta |v|^2}-1)dx \le c_2\|v\|_1^q,
		\end{equation*}
		for some $c_2>0$.
        
		By Sobolev embedding theorem, there is $c_3>0$ such that $|v|_2^2 \le c_3 \|v\|_1^2$. We may choose $\varepsilon$ so that $ \frac{1}{2}-c_3 \varepsilon \ge \frac{1}{4}$.  We obtain
		\begin{align*}
			J(0,v)&\ge \dfrac{1}{2}\|v\|^2_1-\varepsilon c_3\|v\|_1^2-c_2\|v\|_1^q
			\ge \dfrac{1}{4}\|v\|_1^2-c_2\|v\|_1^q.
		\end{align*}
        
	Hence, there exist $\alpha>0$ and we can now choose $\rho>0$ sufficiently small such that 
		\begin{equation*}
			J(0,v)\ge \alpha>0, \;\; \text{whenever}\;\; \|v\|_1 =\rho.
		\end{equation*}
	\end{proof}
	\begin{lemma}
		\label{lem 4.6}
		There is an $(0,e) \in \partial B_1\cap X^+$ and $R> \rho$ such that if 
		\begin{equation*}
			M:= \Big\{(u,0)\in X^-\;|\; \|(u,0)\|\le R\Big\}\oplus \Big\{r(0,e)\;|\;0\le r\le R\Big\},
		\end{equation*}
		then
		\begin{equation*}
			\underset{(u,re)\in \partial M}{\sup}\; J(u,re)\le 0.
		\end{equation*}		
	\end{lemma}
	\begin{proof}
		On $X^-$, since $H(x,u,v)\ge 0$, we have 
		\begin{equation*}
				J(u,0)= -\dfrac{1}{2}\|u\|_1^2 -\int_{\mathbb{R}^2} H(x,u,0)dx \le 0.
		\end{equation*}
		Let $(0,e)\in X^+$ such that $\|(0,e)\|=1$. By Lemma $\ref{lem 4.4}$, there exists $ C_0>0$ such that 
			$H(x,u,v) \ge C_0 (|u|^{\mu}+|v|^\mu)$. For some $\bar{C_0}>0$, we have 
		\begin{align*}
			J\Big((u,0)+r(0,e)\Big) =J(u, re)&=\dfrac{1}{2}\|(0,re)\|^2 -\dfrac{1}{2}\|(u,0)\|^2 -\int_{\mathbb{R}^2}H(x,u, re)dx\\
			&\le \dfrac{1}{2}r^2 -\dfrac{1}{2}\|u\|_1^2 -C_0|u|_\mu^\mu -C_0r^\mu |e|_\mu^\mu\\
            & \le \dfrac{1}{2}r^2 -\dfrac{1}{2}\|u\|_1^2 -\bar{C_0}r^\mu
		\end{align*}

		Since $\mu >2$, it follows that 
		\begin{equation*}
			J(u, re) \rightarrow -\infty, \;\text{whenever}\;\|(u, re)\| \rightarrow +\infty,
		\end{equation*}
		and so, for some $\overline{R}>\rho$ sufficiently large, for every $w \in X^-\oplus \mathbb{R}(0,e)$ such that $\|w\| > \overline{R}$, we have  $J(w)\le 0$ .
	\end{proof}
	
		\subsection{Proof of Theorem $\ref{theo 4.1}$} Here we prove Theorem $\ref{theo 4.1}$. 
        \begin{proof}
        The functional $J$ given in $(\ref{eq 11})$, associated with problem $(\ref{eq 1})$, is of class $\mathcal{C}^1$, $\tau-$upper semi-continuous, and its gradient $\nabla J$ is weakly sequentially continuous by Lemma $\ref{lem 4.3}$. Hence, $J$ satisfies assumption $(A)$ of Theorem $\ref{theorem 3.1}$. Moreover, the geometric assumptions $(J_1)$ and $(J_2)$ of Theorem $\ref{theorem 3.1}$ are fulfilled by Lemma $\ref{lem 4.5}$ and Lemma $\ref{lem 4.6}$, respectively.
        
	By Theorem $\ref{theorem 3.1}$, for some $c\in \mathbb{R}$, there exists $\Big\{(u_n,v_n)\Big\}\subset X$ such that 
	\begin{equation}
		\label{eq 15}
		J(u_n,v_n) \rightarrow c\ge \alpha>0, \;\; J'(u_n,v_n) \rightarrow 0, \;\; \text{as}\;\; n\to \infty.
	\end{equation}
    
	From $(\ref{eq 15})$, the sequence $(u_n,v_n)\subset X$ satisfies (see \cite{FDR} or \cite{ZL})
	\begin{equation}
		\label{eq 16}
	J(u_n,v_n)=c+\delta_n, \;\; \Big|J'(u_n,v_n)(w,z)\Big|\le\varepsilon_n\|(u_n,v_n)\|,
	\end{equation}
	where 
		$w$, $z$ $\in\Big\{u_n, v_n\Big\},\;\; \delta_n\to 0,\;\;\varepsilon_n\to 0, \;\; \text{as}\;\; n\to \infty$.
    
 Taking $(w,z)=(u_n,v_n)$ in $(\ref{eq 16})$, we have 
	\begin{equation*}
		-J'(u_n,v_n)(u_n,v_n)\le \Big|J'(u_n,v_n)(u_n,v_n)\Big| \le\varepsilon_n \|(u_n,v_n)\|,
	\end{equation*}
	and
	\begin{align*}
	\int_{\mathbb{R}^2} \Big(u_n H_u(x,u_n,v_n)+v_nH_v(x,u_n,v_n)\Big)dx&= -J(u_n,v_n)(u_n,v_n)+\|v_n\|_1^2- \|u_n\|_1^2\\
	&= -J(u_n,v_n)(u_n,v_n) +2J(u_n,v_n)+2 \int_{\mathbb{R}^2}H(x,u_n,v_n)dx\\
	&\le C_0+2\delta_n +\varepsilon_n\|(u_n,v_n)\| \\ &+\dfrac{2}{\mu} \int_{\mathbb{R}^2}\Big(u_n H_u(x,u_n,v_n)+v_nH_v(x,u_n,v_n)\Big)dx,
	\end{align*}
	where $C_0=2c$ and $\mu$ is in assumption $(H_3)$.
	Since $\mu>2$, we have $1-\frac{2}{\mu}>0$, and thus, for every $n\in \mathbb{N}$,
	\begin{equation}
		\label{eq 17}
	\Bigg(1-\dfrac{2}{\mu}\Bigg) \int_{\mathbb{R}^2} \Big(u_n H_u(x,u_n,v_n)+v_nH_v(x,u_n,v_n)\Big)dx \le C_0 +2\delta_n +\varepsilon_n\|(u_n,v_n)\|.
	\end{equation}
	
	On the other hand, let $(w,z)=(0,v_n)$, $(w,z)=(u_n,0)$ in $(\ref{eq 16})$. Then, we have 
	\begin{equation*}
		\|v_n\|_1^2-\int_{\mathbb{R}^2} v_n H_v(x,u_n,v_n)dx= J'(u_n,v_n)(0,v_n)\le \varepsilon_n\|v_n\|_1,
	\end{equation*}
	and \begin{equation*}
			\|u_n\|_1^2-\int_{\mathbb{R}^2} u_n H_u(x,u_n,v_n)dx\le \Big|-J'(u_n,v_n)(u_n,0)\Big|\le \varepsilon_n\|u_n\|_1.
	\end{equation*}
	That is, 
	\begin{equation}
		\label{eq 18}
	\|v_n\|_1 \le \int_{\mathbb{R}^2} H_v(x,u_n,v_n)\dfrac{v_n}{\|v_n\|}dx+\varepsilon_n, \;\; \|u_n\|_1 \le \int_{\mathbb{R}^2} H_u(x,u_n,v_n)\dfrac{u_n}{\|u_n\|}dx+\varepsilon_n.
	\end{equation}
    
	We now rely on the following inequality, whose proof is given in \cite[Lemma 2.4]{FDR},
	\begin{equation}
		\label{eq 19}
		st\le
		\begin{cases}
		(e^{t^2} - 1) + s (\log s)^{\frac{1}{2}}, \;\; \text{for all } t\ge0 \text{ and } s \ge e^\frac{1}{4},\\
		(e^{t^2} - 1) +\dfrac{1}{2}s^2, \;\; \text{for all } t\ge0 \text{ and } 0\le s \le e^\frac{1}{4}.
		\end{cases}
	\end{equation}
    
By assumption $(H_3)$, $v$ and $H_v$ have the same sign, as do $u$ and $H_u$. If $v_n>0$, then
	set \[t=V_n:= \dfrac{v_n}{\|v_n\|}\;\; \text{and}\;\; s =\dfrac{H_v(x,u_n,v_n)}{K},\]  where $K$ is the constant appearing in $(\ref{eq 2})$. Otherwise, set \[t=-V_n\;\; \text{and} \;\; s=-\dfrac{H_v(x,u_n,v_n)}{K}.\] 
    
    Without loss of generality, we may assume that $t=V_n$ and $s=\dfrac{H_v(x,u_n,v_n)}{K}$. We have 
	\begin{multline}
		K\int_{\mathbb{R}^2} \dfrac{H_v(x,u_n,v_n)}{K} V_n \,dx \le K \Bigg(\int_{\mathbb{R}^2} \Big(e^{\beta V_n^2}-1\Big)dx\Bigg)\\ +K \int_{\Big\{x\in \mathbb{R}^2\;\Big|\;\dfrac{H_v(x,u_n,v_n)}{K} \ge e^{\frac{1}{4}}\Big\}} \dfrac{H_v(x,u_n,v_n)}{K} \Bigg(\log \dfrac{H_v(x,u_n,v_n)}{K}\Bigg)^{\frac{1}{2}}dx \\+\dfrac{K}{2}\int_{\Big\{x\in \mathbb{R}^2\;\Big|\;\dfrac{H_v(x,u_n,v_n)}{K} \le e^{\frac{1}{4}}\Big\}} \Bigg(\dfrac{H_v(x,u_n,v_n)}{K}\Bigg)^2dx.
	\end{multline}
    Let us set $\mathcal{E}:= \Big\{x\in \mathbb{R}^2\;\Big|\;\dfrac{H_v(x,u_n,v_n)}{K} \le e^{\frac{1}{4}}\Big\}$ and $\mathcal{T}:= \Big\{x\in \mathbb{R}^2\;\Big|\;\dfrac{H_v(x,u_n,v_n)}{K} \ge e^{\frac{1}{4}}\Big\}$.\\
	By $(\ref{eq 12})$, 
	\begin{equation*}
		\int_{\mathbb{R}^2} \Big(e^{\beta V_n^2}-1\Big)dx <+\infty,
	\end{equation*}
	and by $(\ref{eq 2})$,
	\begin{equation}
    \label{eq 21}
		\Bigg(\log \dfrac{H_v(x,u_n,v_n)}{K}\Bigg)^{\frac{1}{2}} \le \beta^{\frac{1}{2}}v_n.
	\end{equation}
	By assumptions $(H_1)-(H_2)$, there exist constants $c_0>0$ and $s_0>0$ such that $H_v(x,u,v)\le c_0v$ for $v\in[0,s_0]$ and $x\in \mathbb{R}^2$, and hence,
    \[H_v(x,u,v) \le c_0v,\;\;\text{for}\;\; \Big\{ x\in \mathbb{R}^2\,\Big|\, v\in[0,s_0]\;\;\text{and}\;\; \dfrac{H_v(x,u,v)}{K} \le e^{\frac{1}{4}} \Big\},\]
    while for $v>s_0$,
    \[H_v(x,u,v) \le \dfrac{e^{1/4}}{s_0}v,\;\;\text{for}\;\; \Big\{ x\in \mathbb{R}^2\,\Big|\, v\in(s_0, +\infty)\;\;\text{and}\;\; \dfrac{H_v(x,u,v)}{K} \le e^{\frac{1}{4}} \Big\}.\]
    Consequently,
    \begin{align}
    \label{eq 22}
        \dfrac{K}{2}\int_{\mathcal{E}}\Bigg(\dfrac{H_v(x,u_n,v_n)}{K}\Bigg)^2dx&=\dfrac{1}{2K}\int_{\mathcal{E}} H_v(x,u_n,v_n) H_v(x,u_n,v_n)\,dx\nonumber \\
        &\le \dfrac{1}{2K}\int_{\mathcal{E}} H_v(x,u_n,v_n)\Big( c_0v_n +\dfrac{e^{1/4}}{s_0}v_n\Big)\,dx \nonumber\\
        &\le C\int_{\mathcal{E}} H_v(x,u_n,v_n)v_n \,dx,
    \end{align}
    for some constant $C>0$. \\	From $(\ref{eq 21})$ and $(\ref{eq 22})$, we have 
	\begin{equation*}
	\int_{\mathbb{R}^2} H_v(x,u_n,v_n)\dfrac{v_n}{\|v_n\|}\, dx \le K_1 + \beta^{\frac{1}{2}}\int_{\mathcal{T}} H_v(x,u_n,v_n)v_n\, dx + C\int_{\mathcal{E}} H_v(x,u_n,v_n)v_n\, dx,
	\end{equation*}
	for some constant $K_1>0$.\\
	Using $(\ref{eq 18})$, we have 
	\begin{equation}
		\label{eq 23}
		\|v_n\|_1 \le K_1+ C\beta^{\frac{1}{2}}\int_{\mathbb{R}^2} H_v(x,u_n,v_n)v_n \,dx +\varepsilon_n.
	\end{equation}
    
	A similar argument shows that 
	\begin{equation}
		\label{eq 24}
		\|u_n\|_1 \le K_2+ C\beta^{\frac{1}{2}}\int_{\mathbb{R}^2} H_u(x,u_n,v_n)u_n \,dx +\varepsilon_n,
	\end{equation}
	for some constant $K_2>0$. Combining $(\ref{eq 23})$, $(\ref{eq 24})$ and $(\ref{eq 17})$, we obtain
	\begin{equation*}
		\|(u_n,v_n)\| \le D_0 +2D_1\delta_n +\varepsilon_n\Big(D_1\|(u_n,v_n)\| +2\Big),
	\end{equation*}
	for some positive constants $D_0$ and $D_1$.\\ It follows that $(u_n,v_n)$ is bounded. Thus, for a subsequence still denoted by $(u_n,v_n)$, there exists $(u_0,v_0)\in X$ such that 
    \begin{center}
	$(u_n,v_n) \to (u_0,v_0) \;\; \text{weakly in}\;\;X, \;\;\text{as}\;\; n \to \infty$.
	\end{center}
    For any $(w,z)\in \mathcal{C}_0^\infty(\mathbb{R}^2)\times \mathcal{C}_0^\infty(\mathbb{R}^2)$, since 
    $(u_n,v_n) \to (u_0,v_0), \;\; \text{in}\;\; L_{loc}^2(\mathbb{R}^2) \times L_{loc}^2(\mathbb{R}^2),\;\; \text{as}\;\; n \to \infty$, and
    $H_u(x,u_n,v_n)$, $H_v(x,u_n,v_n)$ belong to $L^1(\Omega)$, and 
    \[\int_\Omega |u_nH_u(x,u_n,v_n)|\,dx\le \bar{c}_1, \;\; \int_\Omega |u_vH_u(x,u_n,v_n)|\,dx\le \bar{c}_2,\]
    with $\Omega = supp \,\{(w,z)\}\subset\mathbb{R}^2$, then by Lemma $\ref{lem 4.2}$, we deduce 
    \[J'(u_0,v_0)(w,z) =\lim_{n\to \infty} J'(u_n,v_n)(w,z) =0.\]
Hence, $J'(u_0,v_0)(w,z)=0$ for every $(w,z)\in X$, that is, $(u_0,v_0)$ is weak solution of problem $(\ref{eq 1})$.
	
	\begin{claim}
	\textbf{The point $(u_0,v_0)$ is nontrivial.}
	\end{claim}
	Indeed, by contradiction, suppose that $(u_0,v_0)=(0,0)$.  By Hölder inequality, we have 
		\begin{equation*}
	\int_{\mathbb{R}^2} H_u(x,u_n,v_n)u_n\,dx\le \Big(\int_{\mathbb{R}^2} |u_n|^t\,dx\Big)^{1/t}\Big(\int_{\mathbb{R}^2} |H_u(x,u_n,v_n)|^{t'} \,dx\Big)^{1/t'},
		\end{equation*}
		where $\frac{1}{t}+\frac{1}{t'}=1$.\\
        Since $H_u$ has a subcritical growth (see $(\ref{eq 2})$),
        \begin{align*}
            \Big(\int_{\mathbb{R}^2} |H_u(x,u_n,v_n)|^{t'} \,dx\Big)^{1/t'} &\le \Big(\int_{\mathbb{R}^2} K^{t'}\big(e^{\beta u_n^2}-1\big)^{t'} \,dx\Big)^{1/t'}\\
            &\le K  \Big(\int_{\mathbb{R}^2} \big(e^{\beta u_n^2}-1\big)^{t'} \,dx\Big)^{1/t'}.
        \end{align*}
        By \cite[Lemma 2.2]{dMS},
        \begin{equation*}
            \int_{\mathbb{R}^2} \big(e^{\beta u_n^2}-1\big)^{t'}\,dx< +\infty.
        \end{equation*}
        We obtain, 
        \begin{equation*}
           \int_{\mathbb{R}^2} H_u(x,u_n,v_n)u_n\,dx\le \kappa\Big(\int_{\mathbb{R}^2} |u_n|^t\,dx\Big)^{1/t},
        \end{equation*}
        for some constant $\kappa>0$. \\
        Since $u_n \rightharpoonup 0$ in $L^t(\mathbb{R}^2)$ as $n\to \infty$, then 
		\begin{equation*}
		\int_{\mathbb{R}^2} H_u(x,u_n,v_n)u_n\,dx \to 0, \;\; \text{as}\;\; n\to \infty.
		\end{equation*}
		In the same way, we have 
			\begin{equation*}
			\int_{\mathbb{R}^2} H_v(x,u_n,v_n)v_n\,dx \to 0, \;\; \text{as}\;\; n\to \infty.
		\end{equation*}
		So, by assumption $(H_3)$, \begin{equation}
			\label{eq 25}
			\int_{\mathbb{R}^2}H(x,u_n,v_n)\,dx \to 0, \;\; \text{as}\;\; n\to \infty.
		\end{equation}
        
	Thus, $(\ref{eq 25})$ with the fact that 
		\begin{equation*}
			J'(u_n,v_n)(u_n,v_n) \to 0, \;\; \text{as}\;\; n\to \infty,
		\end{equation*}
		imply 
		\begin{equation*}
			J(u_n,v_n) \to 0, \;\; \text{as}\;\; n\to \infty.
		\end{equation*}
		This contradicts the fact that $J(u_n,v_n)\to c\ge \alpha> 0, \;\; \text{as}\;\; n\to \infty$.
        
		Consequently, $(u_0,v_0)$ is nontrivial weak solution of problem $(\ref{eq 1})$. The proof of Theorem $\ref{theo 4.1}$ is thus complete. 
        \end{proof}
       \section*{Acknowledgments}
 This work was funded by a  grant from the Natural Sciences and Engineering Research Council of Canada.
	\section*{Disclosure statement}
	The authors declare that they have no financial or non-financial competing interests.

\end{document}